\documentclass{amsart}
\usepackage{indentfirst,enumerate,mathrsfs}
\usepackage{array,amscd}
\usepackage[pdftex]{hyperref}
\usepackage{caption}
\usepackage{multirow}

\def\NN{\mathbb N}

\def\RR{\mathbb R}

\def\HH{\mathcal H}

\def\xx{\mathbf{x}}
\def\yy{\mathbf{y}}
\def\zz{\mathbf{z}}

\newcommand{\dl}{\mathcal{D}(L)}
\newcommand{\dn}{\mathcal{D}(A)\cap\mathcal{D}(L)}
\newcommand{\gp}{g_\rho}

\newcommand{\la} {\lambda}
\newcommand{\al} {\alpha}
\newcommand{\fz}{f_{\mathbf{z},\la}}
\newcommand{\fp}{f_\rho}

\newcommand{\LL}{\mathscr{L}^2(X,\nu;Y)}

\newcommand{\sx}{S_\xx}
\newcommand{\tx}{T_\xx}
\newcommand{\ip}{I_\nu}
\newcommand{\tp}{T_\nu}
\newcommand{\fbar}{\bar f}

\newcommand{\argmin}{\operatornamewithlimits{argmin}}
\newcommand{\paren}[1]{\left(#1\right)}
\newcommand{\brac}[1]{\left\{#1\right\}}

\newcommand{\inner}[1]{\left\langle#1\right\rangle}
\newcommand{\norm}[1]{\left\|#1\right\|}

\newcommand{\tr}{\operatorname{tr}}
\newtheorem{theorem}{Theorem}[section]
\newtheorem{lemma}[theorem]{Lemma}
\newtheorem{corollary}[theorem]{Corollary}
\newtheorem{definition}[theorem]{Definition}
\newtheorem{proposition}[theorem]{Proposition}
\newtheorem{assumption}{Assumption}

\allowdisplaybreaks
\begin{document}
\title[Tikhonov regularization for nonlinear statistical inverse  problems]{Tikhonov regularization with oversmoothing penalty for nonlinear statistical inverse problems} 

\author{Abhishake Rastogi}
\address{Institute of Mathematics, University of Potsdam,
  Karl-Liebknecht-Strasse 24-25, 14476 Potsdam, Germany}
\email{abhishake@uni-potsdam.de}

\date{}

\begin{abstract}
In this paper, we consider the nonlinear ill-posed inverse problem with noisy data in the statistical learning setting. The Tikhonov regularization scheme in Hilbert scales is considered to reconstruct the estimator from the random noisy data. In this statistical learning setting, we derive the rates of convergence for the regularized solution under certain assumptions on the nonlinear forward operator and the prior assumptions. We discuss estimates of the reconstruction error using the approach of reproducing kernel Hilbert spaces. 
\end{abstract}
\maketitle

\section{Introduction}
We consider the nonlinear ill-posed operator equation of the form
\begin{equation*}
A(f) = g
\end{equation*}
with a nonlinear forward operator~$A: \HH \to \HH'$ between the infinite-dimensional Hilbert spaces~$\HH$ and~$\HH'$. Moreover,~$\HH'$ is the space of functions~$g: X\to Y$ for a Polish space~$X$ (the input space) and a real separable Hilbert space~$Y$ (the output space). Ill-posed inverse problems have important applications in the field of science and technology (see, e.g.,~\cite{Engl,Hofmann1986,Schuster,Tikhonov}). 

In classical inverse problem setting, we observe the approximation~$g_\delta$ of the function~$g$ with~$\norm{g-g_\delta}_{\HH'}\leq \delta$ for some known noise level~$\delta$, then we reconstruct the estimator of the quantity~$f$ through the regularization schemes. Here we consider the problem in statistical learning setting in which we observe the random noisy image~$y_i$ at the points~$x_i$. The problem can be described as follows:
\begin{equation}\label{Model}
y_i= g(x_i)+\varepsilon_i,\qquad  g=A(f)
\end{equation} 
where~$\varepsilon_i$ is the random observational noise with~$1\leq i \leq m$ and~$m$ is called the sample size.  

The model~(\ref{Model}) covers nonparametric regression under random design (which we also call the direct problem, i.e.,~$A=I$), and the linear statistical inverse learning problem. Thus, introducing a general nonlinear operator~$A$ gives a unified approach to the different learning problems.  

Suppose the random observations are drawn identically and independently according to the joint probability measure~$\rho$ on the sample space~$Z=X\times Y$ and the probability measure~$\rho$ can be splitting as follows:
\begin{equation*}
\rho(x,y)=\rho(y|x)\nu(x),
\end{equation*}
where~$\rho(y|x)$ is the conditional probability distribution of~$y$ given~$x$ and~$\nu(x)$ is the marginal probability distribution on $X$.

For the statistical inverse problem \eqref{Model}, the goodness of an estimator~$f$ can be measured through the expected risk:
\begin{align}\label{emp.risk}
\mathcal{E}_\rho(f)=\int_Z\norm{A(f)(x)-y}_Y^2d\rho(x,y).
\end{align}

Further, we assume that~$\int_Y\norm{y}_Y^2d\rho(y|x)<\infty$ for any~$x\in X$. Then for the function  
\begin{equation*}
\gp(x)=\int_Y y d\rho(y|x),
\end{equation*}
the expected risk can be expressed as follows:
\begin{align}\label{emp.prop}
\mathcal{E}_\rho(f)=\int_X\norm{A(f)(x)-\gp(x)}_{Y}^2d\nu(x)+\int_Z\norm{\gp(x)-y}_Y^2d\rho(x,y).
\end{align}
Hence we observe that finding the minimizer of the expected risk is equivalent to obtaining the minimizer of the quantity~$\int_X\norm{A(f)(x)-\gp(x)}_{Y}^2d\nu(x)$.  

Since the probability measure~$\rho$ is unknown, the only information of the probability measure is known through the sample. Therefore we use the regularization methods to stably reconstruct the estimator of the quantity $f$. The Tikhonov regularization is widely considered in both the classical inverse problems and the statistical learning theory. We consider the Tikhonov regularization in Hilbert scales which consists of the error term measuring the fitness of data and oversmoothing penalty.  We introduce an unbounded, closed, linear, self-adjoint, strictly positive operator~$L: \dl \subset \HH \to \HH$ with a dense domain of definition~$\dl \subset \HH$ to treat an oversmoothing penalty in terms of a Hilbert scale. For some $\ell > 0$, the operator $L$ satisfies:
\begin{equation}\label{L.op}
\norm{Lf}_{\HH} \geq \ell\norm{f}_{\HH}\qquad \text{for all}\quad f \in \dl . 
\end{equation}  

For a given sample $\zz=\{(x_i,y_i)\}_{i=1}^m$, we define Tikhonov regularization scheme in Hilbert scales: 
\begin{equation}\label{Tikhonov}
\fz=\argmin\limits_{f\in\dn}\brac{\frac{1}{m}\sum\limits_{i=1}^m\norm{ A(f)(x_i)-y_i}_Y^2+\la\norm{L (f-\fbar)}_{\HH}^2}.
\end{equation}
Here~$\fbar \in \dn$ denotes some initial guess of the true solution, which offers the possibility to incorporate a priori information. Here~$\la$ is a positive regularization parameter which controls the trade-off between the error term and the complexity of the solution.

In many practical problems, the operator~$L$ which influences the properties of the regularized approximation is chosen to be a differential operator in some appropriate function spaces, e.g., the space of square-integrable functions~$\LL$. It is well-known that the standard Tikhonov regularization suffers the saturation effect. The finite qualification of Tikhonov regularization can be overcome using the Hilbert scales. The problem \eqref{Tikhonov} is non-convex, therefore the minimizer may not exist in general. For the continuous and weakly sequentially closed\footnote{i.e., if a sequence~$(f_m)_{m\in\NN} \subset \mathcal{D}(A)$ converges weakly to some~$f \in \HH$ and if the sequence~$(A(f_m))_{m\in\NN}\subset \LL$ converges weakly to some~$g\in \LL$, then~$f\in\mathcal{D}(A)$ and~$A(f)=g$.} operator~$A$, there exists a global minimizer of the functional in~(\ref{Tikhonov}). But it is not necessarily unique since~$A$ is nonlinear (see~\cite[Section~4.1.1]{Schuster}). 

Generally, in the classical inverse problem literature (see~\cite{Bissantz2004,Engl,Hohage,Schuster} and references therein), the 2-step approaches are considered in which first they construct the estimator of the function~$g$ by~$g_\delta$ from the observations~$\{(x_i,y_i)\}_{i=1}^m$, then estimate the quantity~$f$ stably using the various regularization schemes. Here we estimate the quantity~$f$ in a 1-step method using the Tikhonov regularization scheme~\eqref{Tikhonov} in the statistical learning setting. 

Now we review the work in the literature related to the considered problem. Regularization schemes in Hilbert scales are widely considered in classical inverse problems (with deterministic noise)~\cite{Egger2018,Hofmann2018,Mathe2006,Nair2005,Natterer1984,Neubauer92,Tautenhahn1996}.  On the contrary, the inverse problems with random observations are not well-studied. The linear statistical inverse problems are studied in~\cite{Cavalier_2011}, under the assumption that the marginal probability measure~$\nu$ is known which is an unrealistic assumption since the only information is available through the input points~$(x_1,\ldots,x_m)$. This problem is also discussed in~\cite{Blanchard} for the general random design with an unknown marginal probability measure.


In this nonlinear setup, the reference~\cite{OSullivan} established the error estimates for the generalized Tikhonov regularization for~(\ref{Model}) using the linearization technique in a random design setting. In other work,  the authors~\cite{Bissantz2004} consider a 2-step approach, however, again under the assumption of the norm in~$\LL$ being known. The references~\cite{Bauer09} and~\cite{Hohage,Werner2020} consider respectively a Gauss-Newton algorithm and the Tikhonov regularization for certain nonlinear inverse problem, but also in the idealized setting of Hilbertian white or colored noise with known covariance, which can only cover sampling effects when~$\LL$ is known. Loubes et al.~\cite{Loubes} discussed the problem~(\ref{Model}) under a fixed design and concentrate on the problem of model selection. Finally, the recent work~\cite{Abhishake2019} discussed the rates of convergence for the Tikhonov regularization of the nonlinear inverse problem. 

In contrast with the existing work~\cite{Bauer09,Bissantz2004,Hohage,Werner2020} our results are improved in three respects:
\begin{itemize}
\item We do not restrict ourselves to the Hilbertian white or colored noise.

\item We consider a 1-step approach rather than existing 2-step approaches for the nonlinear inverse problems.

\item The considered approach does not suffer the saturation effect of standard Tikhonov regularization.
\end{itemize}

Following the work~\cite{Abhishake2019,Blanchard}, we develop the error analysis for the Tikhonov regularization scheme for the nonlinear inverse problems in Hilbert scales in the statistical learning setting. We establish the error bounds for the statistical inverse problems in reproducing kernel approach. We discuss the rates of convergence for Tikhonov regularization under certain assumptions on the nonlinear forward operator and the prior assumptions. 

Some structural assumptions are required on the nonlinear mappings~$A$ to establish the convergence analysis. We consider the widely assumed conditions in the literature of the classical inverse problems, first assumed in~\cite{Hohage}, and presented in detail in the monograph~\cite{Schuster}. We assume that the operator $A$ is Fr\'echet differentiable at the true solution, the Fr\'echet derivative is Lipschitz continuous and satisfies the link condition (for precise statement see Assumption~\ref{A.assumption}).


The goal is to analyze the theoretical properties of the Tikhonov estimator~$f_{\zz,\la}$, in particular, the asymptotic performance of the regularization scheme is evaluated by the error estimates of the Tikhonov estimator~$f_{\zz,\la}$ in the reproducing kernel approach. Precisely, we develop a non-asymptotic analysis of Tikhonov regularization~(\ref{Tikhonov}) for the nonlinear statistical inverse problem based on the tools that have been developed for the modern mathematical study of reproducing kernel methods. The challenges specific to the studied problem are that the considered model is an {\it inverse problem} (rather than a pure prediction problem) and {\it nonlinear}. The rate of convergence for the Tikhonov estimator~$f_{\zz,\la}$ to the true solution is described in the probabilistic sense by exponential tail inequalities. For sample size~$m$ and the confidence level~$0<\eta<1$, we establish the bounds of the form
$$\mathbb{P}_{\zz\in Z^m}\left\{\norm{f_{\zz,\la}-f}_{\HH}\leq  \varepsilon(m)\log^2\left(\frac{1}{\eta}\right)\right\}\geq 1-\eta.~$$
Here the function~$m\mapsto \varepsilon(m)$ is a positive decreasing function and describes the rate of convergence as~$m\to \infty$.

The paper is organized as follows. In Section~\ref{Sec-Notation}, we discuss the basic definition and assumptions required in our analysis. In Section~\ref{Sec-analysis}, we discuss the bounds of the reconstruction error under certain assumptions on the (unknown) joint probability measure~$\rho$, and the (nonlinear) mapping~$A$. In Appendix, we present the probabilistic estimates and the preliminary results which provide the tools to obtain the error bounds in reproducing kernel approach.   

\section{Notation and assumptions}\label{Sec-Notation}
In this section, we introduce some basic concepts, definitions, and notations required in our analysis. 

\subsection{Reproducing Kernel Hilbert space and related operators}
We start with the concept of the reproducing kernel Hilbert spaces. It is a subspace of~$\LL$ (the space of square-integrable functions from~$X$ to~$Y$ with respect to the probability distribution~$\nu$) which can be characterized by a symmetric, positive semidefinite kernel and each of its function satisfies the reproducing property. Here we discuss the vector-valued reproducing kernel Hilbert spaces~\cite{Micchelli1} which are the generalization of real-valued reproducing kernel Hilbert spaces~\cite{Aronszajn}.

\begin{definition}[Vector-valued reproducing kernel Hilbert space]
For a non-empty set~$X$ and a real separable Hilbert space~$(Y,\langle\cdot,\cdot\rangle_Y)$, a Hilbert space~$\HH$ of functions from~$X$ to~$Y$ is said to be the vector-valued reproducing kernel Hilbert space, if the linear functional~$F_{x,y}:\HH \to \RR$, defined by
$$F_{x,y}(f)=\langle y,f(x)\rangle_Y \qquad \forall f \in \HH,$$
is continuous for every~$x \in X$ and~$y\in Y$.
\end{definition}

Throughout the paper,~$T^*$ denotes adjoint of an operator~$T$.
\begin{definition}[Operator-valued positive semi-definite kernel]
Suppose~$\mathcal{L}(Y):Y\to Y$ is the Banach space of bounded linear operators. A function~$K:X\times X\to \mathcal{L}(Y)$ is said to be an operator-valued positive semi-definite kernel if 
\begin{enumerate}[(i)]
\item~$K(x,x')^*=K(x',x) \qquad\forall~x,x'\in X.$ 

\item~$\sum\limits_{i,j=1}^N\langle y_i,K(x_i,x_j)y_j\rangle_Y\geq 0  \qquad\forall~\{x_i\}_{i=1}^N\subset X \text{ and } \{y_i\}_{i=1}^N\subset Y.$
\end{enumerate}
\end{definition}

For a given operator-valued positive semi-definite kernel~$K:X \times X \to \mathcal{L}(Y)$, we can construct a unique vector-valued reproducing kernel Hilbert space~$(\HH,\langle\cdot,\cdot\rangle_{\HH})$ of functions from~$X$ to~$Y$ as follows:
\begin{enumerate}[(i)]
\item  We define the linear function
\[
K_x: Y \rightarrow \HH: y \mapsto K_xy,
\]
where~$K_xy:X \to Y:x' \mapsto (K_xy)(x')=K(x',x)y$  for~$x,x'\in X$ and~$y\in Y$.
\item The span of the set~$\{K_xy:x\in X, y\in Y\}$ is dense in~$\HH$.
\item \textbf{Reproducing property:} \[\langle f(x),y\rangle_Y=\langle f,K_xy\rangle_{\HH},\qquad x\in X,~y \in Y,~\forall~f\in \HH,\]
 in other words~$f(x) = K_x^* f$.
\end{enumerate}

Moreover, there is a one-to-one correspondence between operator-valued positive semi-definite kernels and vector-valued reproducing kernel Hilbert spaces~\cite{Micchelli1}. The reproducing kernel Hilbert space becomes real-valued reproducing kernel Hilbert space, in the case that~$Y$ is a bounded subset of~$\RR$, and the corresponding kernel becomes the symmetric, positive semi-definite~$K:X \times X \to \RR$ with the reproducing property~$f(x)=\langle f,K_x\rangle_{\HH}$.  
  
We assume the following assumption concerning the Hilbert space~$\HH'$:
\begin{assumption} \label{assmpt1} 
The space~$\HH'$ is assumed to be a vector-valued reproducing kernel Hilbert space of functions~$f:X\to Y$ corresponding to the kernel~$K:X\times X\to \mathcal{L}(Y)$ such that
  \begin{enumerate}[(i)]
  \item~$K_x:Y\to\HH'$ is a Hilbert-Schmidt
    operator for~$x\in X$ with
    \[\kappa^2:=\sup_{x \in X} \norm{K_x}^2_{HS} = {\sup_{x \in
          X}\tr(K_x^*K_x)}<\infty.\]
  \item For~$y,y'\in Y$, the real-valued function~$\varsigma:X\times X \to \RR:(x,x')\mapsto\langle K_{x}y,K_{x'}y'\rangle_{\HH'}$ is measurable.
  \end{enumerate}
\end{assumption}

Note that in case of real-valued functions ($Y\subset\RR$), Assumption~\ref{assmpt1} simplifies to the condition that the kernel is measurable and~$\kappa^2:=\sup_{x \in X} \norm{K_x}^2_{\HH'}=\sup_{x \in X}K(x,x)<\infty$.

Now we introduce some relevant operators used in the convergence analysis. We introduce the notations for the discrete ordered sets~$\xx=(x_1,\ldots,x_m)$,~$\yy=(y_1,\ldots,y_m)$,~$\zz=(z_1,\ldots,z_m)$. The product Hilbert space~$Y^m$ is equipped with the inner product~$\inner{\yy,\yy'}_m = \frac{1}{m}\sum_{i=1}^m \inner{y_i,y'_i}_Y,$ and the corresponding norm~$\norm{\yy}^2_m=\frac{1}{m}\sum_{i=1}^m\norm{y_i}_Y^2$. We define the {\it sampling operator}~$\sx:\HH' \to Y^m:g\mapsto(g(x_1),\ldots,g(x_m))$, then the adjoint~$\sx^*:Y^m\to\HH'$ is
given by
$$\sx^*\mathbf{c}=\frac{1}{m}\sum_{i=1}^m K_{x_i} c_i,~~~~\forall \mathbf{c}=(c_1,\ldots,c_m)\in Y^m.$$  

Let~$\ip$ be the canonical injection map $\HH'$ to $\LL$. Then we observe that both the canonical injection map~$\ip$ and the sampling operator ~$\sx$ are bounded by~$\kappa$ under Assumption~\ref{assmpt1}, since
\[\norm{\ip f}^2_{\LL}=\int_X\norm{f(x)}_Y^2d\nu(x)
  =\int_X\norm{K_x^*f}_{Y}^2d\nu(x) \leq
  \kappa^2\norm{f}^2_{\HH}
\]
and
 \[ \norm{\sx f}^2_m 
   =\frac{1}{m}\sum_{i=1}^m\norm{f(x_i)}_Y^2
   =\frac{1}{m}\sum_{i=1}^m\norm{K_{x_i}^*f}_Y^2 \leq
   \kappa^2\norm{f}^2_{\HH}.\]

We denote the population version~$\tp=\ip^{\ast}\ip\colon \HH'\to \HH'$, the corresponding covariance operator. The operator~$\tp$ is positive, self-adjoint and depends on both the kernel and the marginal probability measure $\nu$. We also introduce the sampling version operator $\tx=\sx^*\sx$ which is positive, self-adjoint and depends on both the kernel and the inputs $\xx$.

By the spectral theory, the operator~$L^s : \mathcal{D}(L^s) \to \HH$ is well-defined for~$s \in \RR$, and the spaces~$\HH_s := \mathcal{D}(L^s), s \geq 0$ equipped with the
inner product~$\langle f,g \rangle_{\HH_s}=\langle L^s f,L^s g\rangle_\HH,~~f, g \in\HH_s$ are Hilbert spaces. For~$s < 0$, the spaces~$\HH_s$ is defined as completion of~$\HH$ under the norm
$\norm{f}_s := \inner{f, f}_s^{1/2}$. The space~$(\HH_s)~s\in\RR$ is called the Hilbert scale induced by~$L$. We notice that the space~$\HH_0$ is~$\HH$ according to the above notations. The interpolation inequality is an important tool for the analysis:
\begin{equation}\label{interpolation}
\norm{f}_{\HH_r}\leq\norm{f}_{\HH_t}^{\frac{s-r}{s-t}}\norm{f}_{\HH_s}^{\frac{r-t}{s-t}},\qquad f\in \HH_s
\end{equation}
which holds for any~$t < r < s$.  

\subsection{The true solution, noise condition, and nonlinearity structure}
We consider that random observations~$\{(x_i,y_i)\}_{i=1}^m$ follow the model~$y= A(f)(x)+\varepsilon$ with a centered noise~$\varepsilon$.

We assume throughout the paper that the operator $A$ is injective.


\begin{assumption}[The true solution]\label{fp}
The conditional expectation w.r.t.~$\rho$ of~$y$ given~$x$ exists (a.s.), and there exists~$\fp \in \textnormal{int}(\mathcal{D}(A)) \subset\HH~$ such that
  \begin{equation*}
\int_Y y d\rho(y|x) = \gp(x) = A(\fp)(x), \text{ for all } x\in X.
  \end{equation*}
\end{assumption} 

From~\eqref{emp.prop} we observe that~$\fp$ is the minimizer of the  expected risk. The element~$\fp$ is the true solution which we aim at estimating.

\begin{assumption}[Noise condition]\label{noise.cond}
There exist some constants~$M,\Sigma$ such that for almost all~$x\in X$,
\begin{equation*}
\int_Y\left(e^{\norm{y-A(\fp)(x)}_Y/M}-\frac{\norm{y-A(\fp)(x)}_Y}{M}-1\right)d\rho(y|x)\leq\frac{\Sigma^2}{2M^2}.
\end{equation*}
\end{assumption}

This Assumption is usually referred to as a \emph{Bernstein-type assumption}. The distribution of the observational noise reflects in terms of the parameters $M > 0$, $\Sigma > 0$. For the convergence analysis, the output space need not be bounded as long as the noise condition for the output variable is fulfilled.

We need the assumption on the nonlinearity structure of operator $A$ to establish the rates of convergence. Following the work of Engl et al.~ \cite[Chapt.~10]{Engl},~\cite{Hohage} on `classical' nonlinear inverse problems, we consider the following assumption:
\begin{assumption}[nonlinearity structure]\label{A.assumption}
\begin{enumerate}[(i)]
  \item~$\mathcal{D}(A)$ is convex,~$A:\mathcal{D}(A)\cap \dl \to \HH'$ is weakly sequentially closed and~$A$ is Fr\'{e}chet differentiable with derivative~$A':\mathcal{D}(A)\to L(\HH,\HH')$.

\item the Fr{\'e}chet derivative~$A'(f)$ is bounded in a ball of sufficiently large radius~$d$, i.e., there exists
$J < \infty$ such that
\begin{equation*}
\norm{A'(f)}_{\HH\to\HH'}\leq J\qquad\forall f \in \mathcal{B}_d(\fp ) \cap \mathcal{D}(A) \subset \HH,
\end{equation*}
      
  \item (Link condition) There exists constants~$p>0$ and~$\al,\beta>0$ such that for all~$g\in \HH$,
   \begin{equation*}
  \al\|g\|_{\HH_{-p}}\leq \|\ip  A'(\fp)g\|_{\LL}\leq \beta\|g\|_{\HH_{-p}}.
   \end{equation*}
  
  \item (Lipschitz continuity of $A'$) For all~$f\in \mathcal{D}(A)\cap\dl~$, there exists a constant~$\gamma$ such that
~$$\|\ip  \brac{A'(\fp)-A'(f)}\|_{ \HH_{-p}\to\LL}\leq \gamma\|\fp-f\|_{\HH}\leq \frac{\al^2}{2\beta}.$$
\end{enumerate}
\end{assumption}

A sufficient condition for weak sequential closedness is that $\mathcal{D}(A)$ is weakly closed (e.g. closed and convex) and $A$ is weakly continuous. The link condition (Assumption~\ref{A.assumption} (iii)) is an interplay between the operator~$L^{-1}$ and the Fr\'echet derivative of the operator~$A$. This link condition is known as finitely smoothing. This condition is satisfied in various types of  problems (for examples see~\cite[Example~10.2]{Bottcher2006},~\cite[Example~4,~5]{Werner2020}). 

\subsection{Effective dimension}
Now we introduce the concept of the effective dimension which is an important ingredient to derive the rates of convergence rates~\cite{Blanchard,Caponnetto,Guo,Lin2018,Lu2020,Rastogi}.  The effective dimension is defined as
$$\mathcal{N}(\la):=Tr\left((\tp+\la I)^{-1}\tp\right), \text{  for }\la>0.$$
Using the singular value decomposition~$\tp=\sum\limits_{i=1}^\infty t_i\langle\cdot,e_i\rangle_{\HH'} e_i$ for an orthonormal sequence~$(e_i)_{i\in \NN}$ of eigenvectors of~$\tp$ with corresponding eigenvalues~$(t_i)_{i\in\NN}$ such that~$t_1\geq t_2\geq\ldots\geq 0$, we get
$$\mathcal{N}(\la)=\sum\limits_{i=1}^\infty \frac{t_i}{\la+t_i}.$$
Hence the function~$\la\to \mathcal{N}(\la)$ is continuous and decreasing from~$\infty$ to zero for~$0 < \la < \infty$ for the infinite-dimensional operator~$\tp$ (see for details~\cite{Blanchard2012,Blanchard2020,Lin2015,Lu2020,Zhang}). 

Since the integral operator~$\tp$ is a trace class operator, the effective dimension is finite and we have that
\begin{equation}\label{Nl.bd}
\mathcal{N}(\la)\leq \norm{(\tp+\la I)^{-1}}_{\mathcal{L}(\HH')}Tr\left(\tp\right) \leq \frac{\kappa^2}{\la}.
\end{equation}

\begin{assumption}[Polynomial decay condition]\label{N(l).bound}
Assume that there exists some positive constant~$c>0$ such that
\begin{equation*}
\mathcal{N}(\la) \leq c\la^{-b},\qquad\text{ for }\quad b<1,\quad\forall \quad\la>0.
\end{equation*}
\end{assumption}

\begin{assumption}[Logarithmic decay condition]\label{log.decay} Assume that there exists some positive constant~$c>0$ such that
\begin{equation*}
\mathcal{N}(\la)\leq c\log\left(\frac{1}{\la}\right),\qquad\forall \quad\la>0.
\end{equation*}
\end{assumption}

Lu et al.~\cite{Lu2020} showed that different kernels with some probability measures show different behavior of the effective dimension. For Gaussian kernel~$K_1(x,x') = xx' + e^{-8(x-x')^2}$ with the uniform sampling on~$[0,1]$, the effective dimension exhibits the log-type behavior (Assumption~\ref{log.decay}), on the other hand, the kernel~$K_2(x,x') = \min\{x,x'\}-xt$ exhibits the power-type behavior (Assumption~\ref{N(l).bound}).

Caponnetto et al~\cite{Caponnetto} showed that if the eigenvalues~$t_n$'s of the integral operator~$\tp$ follow the polynomial decay: i.e., for fixed positive constants~$\mu$ and~$b<1$,
$$t_n\leq\mu n^{-\frac{1}{b}}\quad\forall n\in\NN,$$ 
then the effective dimension behaves like power-type function (Assumption~\ref{N(l).bound}). 


\section{Convergence analysis}\label{Sec-analysis}

Here we establish the error bounds for the Tikhonov regularization for the nonlinear statistical inverse problems in the~$\HH$-norm in the probabilistic sense. The explicit expression of~$f_{\zz,\la}$ is not known, therefore we use the definition~\eqref{Tikhonov} of the Tikhonov estimator~$f_{\zz,\la}$ to derive the error estimates. The linearization techniques is used for nonlinear operator~$A$ in the neighborhood of the true solution~$\fp$. The rates of convergence are established by exploiting the nonlinearity structure of operator $A$ (see~Assumption~\ref{A.assumption}). We discuss the rates of convergence for the Tikhonov estimator by measuring the effect of random sampling which is governed by the noise condition (Assumption~\ref{noise.cond}).  The bounds of the reconstruction error depend on the effective dimension, the smoothness parameter $q$ of the true solution and the parameter $p$ related to the link condition.  


It is convenient to introduce the ``standardized” quantities used in our analysis. Here we introduce shorthand notation for some key quantities. We let
$$\Xi_{\xx}:=\sx (\tx +\la I)^{-1}\sx^*,$$
$$\Delta_{\zz}:=\sx A(\fp )-\yy,$$
$$\Theta_{\zz}:=\norm{(\tp +\la I)^{-1/2}\sx^*\brac{\sx A(\fp )-\yy}}_{\HH'},$$
$$\Psi_\xx:=\norm{(\tp +\la I)^{-1/2}(\tp -\tx  )}_{\mathcal{L}(\HH')}$$ 
and
$$\Gamma_{\xx}:=\norm{(\tx  +\la I)^{-1/2}(\tp +\la I)^{1/2}}_{\mathcal{L}(\HH')}.$$ 

The error bound discussed in the following theorem holds non-asymptotically, but this holds with the following choice of the regularization parameter~$\la$ and sample size~$m$. We can choose appropriate regularization parameter~$\la$ and sample size~$m$ such that the following holds:
\begin{equation}\label{l.la.condition}
\mathcal{N}(\la) \leq m\la\qquad \text{and}\qquad \la\leq \min\paren{1,\norm{\tp}_{\mathcal{L}(\HH')}}.
\end{equation}

The condition (\ref{l.la.condition}) says that as the regularization parameter $\la$ decreases, the sample size $m$ must increase.

\begin{theorem}\label{err.upper.bound.p}
Let~$\zz$ be i.i.d. samples drawn according to the probability measure~$\rho$.  If Assumptions \ref{assmpt1}--\ref{A.assumption}  and \eqref{l.la.condition} hold true and if~$\fp -\fbar\in \HH_q$ for some~$q \in[1,2 + p]$. Then, for the Tikhonov estimator~$\fz~$ in (\ref{Tikhonov}) with the a-priori choice of the regularization parameter~$\la=\Theta_{\mathcal{N},p,q}^{-1}\paren{\frac{1}{\sqrt{m}}}$ for~$\Theta_{\mathcal{N},p,q}(\la)=\frac{\la^{\frac{p+q}{2(p+1)}}}{\sqrt{\mathcal{N}(\la)}}$, for all~$0<\eta<1$, the following error bound holds with the confidence~$1-\eta$:
$$\|\fz -\fp\|_{\HH}=\mathcal{O}\paren{\paren{\Theta_{\mathcal{N},p,q}^{-1}\paren{\frac{1}{\sqrt{m}}}}^{r}\log^2\paren{\frac{4}{\eta}}}\qquad {for}\quad r=\frac{q}{2(p+1)}.$$
\end{theorem}

\begin{proof}

By the definition of~$\fz~$ as the solution of minimization problem in~\eqref{Tikhonov}, we have
\begin{equation*}
\frac{1}{m}\sum\limits_{i=1}^m\|A(\fz )(x_i)-y_i\|_Y^2+\la \|L(\fz -\fbar)\|_{\HH}^2\leq \frac{1}{m}\sum\limits_{i=1}^m\|A(\fp)(x_i)-y_i\|_Y^2+\la \|L(\fp-\fbar)\|_{\HH}^2
\end{equation*}
which implies
\begin{equation}\label{idea}
\|\sx  A(\fz )-\yy\|_m^2+\la \|L(\fz -\fbar)\|_{\HH}^2\leq \|\sx  A(\fp )-\yy\|_m^2+\la \|L(\fp-\fbar)\|_{\HH}^2.
\end{equation}


By linearizing the nonlinear operator~$A$ at~$\fz$ we get 
\begin{equation}\label{Taylor_exp_Rla}
A(\fz )=A(\fp )+A'(\fp )(\fz -\fp )+r(\fz )
\end{equation}
where~$r(\fz)$ is the error term by linearizing the operator~$A$ at true solution~$\fp$. Using this we reexpress the inequality (\ref{idea}) as follows,
\begin{equation*}
\|\sx A'(\fp)(\fz -\fp)+\Delta_{\zz}+\sx  r(\fz)\|_m^2+\la \|L(\fz -\fbar)\|_{\HH}^2\leq \|\Delta_{\zz}\|_m^2+\la \|L(\fp-\fbar)\|_{\HH}^2.
\end{equation*}

Then we have,
\begin{align*}
&\|\sx A'(\fp)(\fz -\fp)\|_m^2+\norm{\Delta_{\zz}+\sx  r(\fz)}_{m}^2+2\langle \sx A'(\fp)(\fz -\fp),\Delta_{\zz}+\sx  r(\fz )\rangle_m   \\  \nonumber
&+\la\|L(\fz -\fp)\|_{\HH}^2 \leq 2\la\langle L(\fp-\fz ),L(\fp-\fbar)\rangle_{\HH}+\norm{\Delta_{\zz}}_{m}^2
\end{align*}
which implies
\begin{align}\label{B1}
&\norm{\ip A'(\fp)(\fz-\fp)}_{\LL}^2+\la\norm{L(\fz-\fp)}_{\HH}^2 \\ \nonumber
\leq  & 2\la\norm{\fz-\fp}_{\HH_{2-q}}\norm{\fp-\fbar}_{\HH_{q}}+\norm{\Delta_{\zz}}_{m}^2-\norm{\Delta_{\zz}+\sx  r(\fz)}_{m}^2  \\  \nonumber
&+\inner{A'(\fp)(\fz-\fp),(\tp -\tx  )A'(\fp)(\fz-\fp)}_{\HH'}-2\inner{A'(\fp)(\fz-\fp),\sx ^*\brac{\Delta_{\zz}+\sx  r(\fz)}}_{\HH'}.
\end{align}

Now with Assumption~\ref{A.assumption} and~\eqref{L.op} from Lemmas~\ref{R_la}--\ref{B3} we obtain,
\begin{align*}
&\al^2\norm{\fz-\fp}_{\HH_{-p}}^2+\la\norm{L(\fz-\fp)}_{\HH}^2  \\
\leq & \delta_1+\sqrt{\la}\delta_2\norm{L(\fz-\fp)}_{\HH}+\delta_3\norm{\fz-\fp}_{\HH_{-p}}+\beta\gamma\norm{\fz-\fp}_{\HH}\norm{\fz-\fp}_{\HH_{-p}}^2\\
&+2\la\norm{\fz-\fp}_{\HH_{2-q}}\norm{\fp-\fbar}_{\HH_{q}},
\end{align*}
where~$\delta_1=\Gamma_{\xx}^2\Theta_{\zz}^2,~\delta_2=\ell\paren{2J\Theta_{\zz}+4J\Gamma_{\xx}\Theta_{\zz}+5J^2\Psi_\xx\norm{\fz-\fp}_{\HH}}$ 
and~$\delta_3=2\beta \Theta_{\zz}+5\beta J\Psi_\xx\norm{\fz-\fp}_{\HH}$.

Under the condition (iii) of Assumption \ref{A.assumption} using the interpolation inequality~\eqref{interpolation}, we obtain
\begin{align}\label{final.eq}
&\frac{\al^2}{2} \norm{\fz -\fp}_{\HH_{-p}}^2+\la\norm{L(\fz -\fp)}_{\HH}^2 \\ \nonumber
\leq & \delta_1+\sqrt{\la}\delta_2\norm{L(\fz-\fp)}_{\HH}+\delta_3\norm{\fz -\fp}_{\HH_{-p}} 
+2\la\norm{\fp-\fbar}_{\HH_{q}}\norm{\fz-\fp}_{\HH_{-p}}^{\frac{q-1}{p+1}}\norm{L(\fz-\fp)}_{\HH}^{\frac{p-q+2}{p+1}}
\end{align}
which can be re-expressed as
\begin{align}\label{f.eq}
&\norm{\fz-\fp}_{\HH_{-p}}^2\\   \nonumber
=&\mathcal{O}\paren{\delta_1+\sqrt{\la}\delta_2\norm{L(\fz-\fp)}_{\HH}+\delta_3\norm{\fz -\fp}_{\HH_{-p}}+\la\norm{\fz-\fp}_{\HH_{-p}}^{\frac{q-1}{p+1}}\norm{L(\fz-\fp)}_{\HH}^{\frac{p-q+2}{p+1}}}.
\end{align}

In the analysis, we will make repeated use of the following:
\begin{equation}\label{imply}
c^r\leq e+dc^t \Rightarrow c^r=\mathcal{O}\paren{e+d^{\frac{r}{r-t}}}
\end{equation}
which holds for~$0 \leq t < r$ and~$c, d, e > 0$. 

We apply this inequality to the estimate~\eqref{f.eq} for~$c = \norm{\fz-\fp}_{\HH_{-p}}$ and~$r = 2$. First we take~$t = 1$,~$d = \delta_3$ and
$e = \delta_1+\sqrt{\la}\delta_2\norm{L(\fz-\fp)}_{\HH}+\la\norm{\fz-\fp}_{\HH_{-p}}^{\frac{q-1}{p+1}}\norm{L(\fz-\fp)}_{\HH}^{\frac{p-q+2}{p+1}}$ and we obtain
\begin{equation*}
\norm{\fz-\fp}_{\HH_{-p}}^2=\mathcal{O}\paren{\delta_1+\delta_3^2+\sqrt{\la}\delta_2\norm{L(\fz-\fp)}_{\HH}+\la\norm{\fz-\fp}_{\HH_{-p}}^{\frac{q-1}{p+1}}\norm{L(\fz-\fp)}_{\HH}^{\frac{p-q+2}{p+1}}}.
\end{equation*}
Then we choose~$t = \frac{q-1}{p+1}$,~$d = \la\norm{L(\fz-\fp)}_{\HH}^\frac{p-q+2}{p+1}$ and~$e = \delta_1+\delta_3^2+\sqrt{\la}\delta_2\norm{L(\fz-\fp)}_{\HH}$ and we get
\begin{equation}\label{ps}
\norm{\fz-\fp}_{\HH_{-p}}^2=\mathcal{O}\paren{\delta_4^2+\sqrt{\la}\delta_2\norm{L(\fz-\fp)}_{\HH}+\la^{\frac{2(p+1)}{2p-q+3}}\norm{L(\fz-\fp)}_{\HH}^{\frac{2(p-q+2)}{2p-q+3}}},
\end{equation}
where~$\delta_4^2=\delta_1+\delta_3^2$.

Replacing the term that contains~$\norm{\fz-\fp}_{\HH_{-p}}$ on the right-hand side in~\eqref{final.eq} and using the inequality~$(x + y)^r \leq x^r + y^r$ for~$0 \leq r \leq 1$ we obtain
\begin{align*}
\norm{L(\fz-\fp)}_{\HH}^2 = & \mathcal{O}\left(\frac{\delta_1}{\la}+\frac{\delta_3\delta_4}{\la}+\frac{\delta_2}{\sqrt{\la}}\norm{L(\fz-\fp)}_{\HH}+\frac{\delta_2^{\frac{1}{2}}\delta_3}{\la^{\frac{3}{4}}}\norm{L(\fz-\fp)}_{\HH}^{\frac{1}{2}}\right.\\
&+\la^{\frac{q-1}{4(p+1)}}\delta_2^{\frac{q-1}{2(p+1)}}\norm{L(\fz-\fp)}_{\HH}^{\frac{2p-q+3}{2(p+1)}}+\delta_3\la^{\frac{q-p-2}{2p-q+3}}\norm{L(\fz-\fp)}_{\HH}^{\frac{p-q+2}{2p-q+3}}  \\  \nonumber
&\left. +\delta_4^{\frac{q-1}{p+1}}\norm{L(\fz-\fp)}_{\HH}^{\frac{p-q+2}{p+1}} +\la^{\frac{q-1}{2p-q+3}}\norm{L(\fz-\fp)}_{\HH}^{\frac{2(p-q+2)}{2p-q+3}}\right).
\end{align*}

Applying (\ref{imply}) repeatedly for~$c = \norm{L(\fz-\fp)}_{\HH}^2$,~$r = 1$ and~$t=\frac{1}{2}$,~$t=\frac{1}{4}$,~$t=\frac{2p-q+3}{4(p+1)}$,~$t =\frac{p-q+2}{2(2p-q+3)}$,~$t=\frac{p-q+2}{2(p+1)}$
and~$t = \frac{p-q+2}{2p-q+3}$ we obtain
\begin{align*}
\norm{L(\fz-\fp)}_{\HH}^2 =\mathcal{O}&\left(  \frac{1}{\la}\brac{\delta_1+\delta_3\delta_4+\delta_2^2+\delta_2^{\frac{2}{3}}\delta_3^{\frac{4}{3}}}+\la^{\frac{q-1}{2p+q+1}}\delta_2^{\frac{2(q-1)}{2p+q+1}}+\la^{\frac{2(q-p-2)}{3p-q+4}}\delta_3^{\frac{2(2p-q+3)}{3p-q+4}}\right.\\
&\left.+\delta_4^{\frac{2(q-1)}{p+q}}+\la^{\frac{q-1}{p+1}}\right),
\end{align*}

\begin{equation}\label{err1}
\norm{L(\fz-\fp)}_{\HH}^2=\mathcal{O}\paren{\frac{\delta^2}{\la}+\la^{\frac{q-1}{2p+q+1}}\delta^{\frac{2(q-1)}{2p+q+1}}+\la^{\frac{2(q-p-2)}{3p-q+4}}\delta^{\frac{2(2p-q+3)}{3p-q+4}}+\delta^{\frac{2(q-1)}{q+p}}+\la^{\frac{q-1}{p+1}}},
\end{equation}
where~$\delta^2=\Psi_\xx^2+\Theta_{\zz}^2+\Gamma_{\xx}^2\Theta_{\zz}^2$. 

Under the condition~\eqref{l.la.condition}, from Propositions~\ref{main.bound}--\ref{I1} we get with the probability~$1-\eta$,
\begin{equation}
\delta= \mathcal{O}\paren{\brac{\frac{1}{m\sqrt{\la}}+\sqrt{\frac{\mathcal{N}(\la)}{m}}}\log^2\paren{\frac{4}{\eta}}}.
\end{equation}

Under the condition~\eqref{l.la.condition} the spectral decomposition of the operator~$\tp~$ gives 
\begin{equation}
\mathcal{N}(\la)\geq \frac{\norm{\tp }_{\mathcal{L}(\HH')}}{\la+\norm{\tp }_{\mathcal{L}(\HH')}} \geq \frac{1}{2} \qquad \text{for}\quad \la\leq \norm{\tp }_{\mathcal{L}(\HH')}.
\end{equation}

From~\eqref{l.la.condition} we get
\begin{equation}
\frac{1}{m\sqrt{\la}}\leq \frac{2\mathcal{N}(\la)}{m\sqrt{\la}} \leq 2\sqrt{\frac{\mathcal{N}(\la)}{m}}.
\end{equation} 

Hence we get,
\begin{equation}
\delta= \mathcal{O}\paren{\sqrt{\frac{\mathcal{N}(\la)}{m}}\log^2\paren{\frac{4}{\eta}}}.
\end{equation}

By balancing the error terms in~\eqref{err1}, we consider the parameter choice~$\la=\Theta_{\mathcal{N},p,q}^{-1}\paren{\frac{1}{\sqrt{m}}}$ for~$\Theta_{\mathcal{N},p,q}(\la)=\frac{\la^{\frac{p+q}{2(p+1)}}}{\sqrt{\mathcal{N}(\la)}}$. We have with the probability~$1-\eta$,
\begin{equation*}
\norm{L(\fz-\fp)}_{\HH}=\mathcal{O}\paren{\la^{\frac{q-1}{2(p+1)}}\log^2\paren{\frac{4}{\eta}}}=\mathcal{O}\paren{\paren{\Theta_{\mathcal{N},p,q}^{-1}\paren{\frac{1}{\sqrt{m}}}}^{\frac{q-1}{2(p+1)}}\log^2\paren{\frac{4}{\eta}}}.
\end{equation*}

Form~\eqref{ps} we get with the probability~$1-\eta$,
\begin{align*}
\norm{\fz-\fp}_{\HH_{-p}}^2 =& \mathcal{O}\paren{\delta^2+\sqrt{\la}\delta\norm{L(\fz-\fp)}_{\HH}+\la^{\frac{2(p+1)}{2p-q+3}}\norm{L(\fz-\fp)}_{\HH}^{\frac{2(p-q+2)}{2p-q+3}}}\\
= & \mathcal{O}\paren{\la^{\frac{p+q}{p+1}}\log^2\paren{\frac{4}{\eta}}}=\mathcal{O}\paren{\paren{\Theta_{\mathcal{N},p,q}^{-1}\paren{\frac{1}{\sqrt{m}}}}^{\frac{p+q}{p+1}}\log^2\paren{\frac{4}{\eta}}}.  
\end{align*}

Taking the mean using the inequality (\ref{interpolation}) we get with the probability~$1-\eta$,
\begin{align*}
\norm{\fz -\fp }_{\HH}\leq & \norm{\fz -\fp }^{\frac{1}{p+1}}_{\HH_{-p}} \norm{L(\fz -\fp )}^{\frac{p}{p+1}}_{\HH}
=\mathcal{O}\paren{\la^{\frac{p+q}{2(p+1)^2}}\la^{\frac{p(q-1)}{2(p+1)^2}}\log^2\paren{\frac{4}{\eta}}}\\
= & \mathcal{O}\paren{\la^{\frac{q}{2(p+1)}}\log^2\paren{\frac{4}{\eta}}}=\mathcal{O}\paren{\paren{\Theta_{\mathcal{N},p,q}^{-1}\paren{\frac{1}{\sqrt{m}}}}^{\frac{q}{2(p+1)}}\log^2\paren{\frac{4}{\eta}}}.
\end{align*}

Now, we obtain the desired result.
\end{proof}
Here we observe that the condition~\eqref{l.la.condition} is automatically satisfied under the a-priori choice of the regularization parameter~$\la=\Theta_{\mathcal{N},p,q}^{-1}\paren{\frac{1}{\sqrt{m}}}$ for~$\Theta_{\mathcal{N},p,q}(\la)=\frac{\la^{\frac{p+q}{2(p+1)}}}{\sqrt{\mathcal{N}(\la)}}$.

Using the trivial bound \eqref{Nl.bd} for the effective dimension we get the following corollary.
\begin{corollary}\label{err.upper.bound.para}
Under the same assumptions of Theorem \ref{err.upper.bound.p} with a-priori choice of the regularization parameter~$\la=m^{-\frac{p+1}{2p+q+1}}$, for all~$0<\eta<1$, the following error estimate holds with confidence~$1-\eta$:
$$||\fz -\fp||_\HH = \mathcal{O}\paren{m^{-\frac{q}{2(2p+q+1)}}\log^2\left(\frac{4}{\eta}\right)}.$$
\end{corollary}

The ill-posedness of the problem is measured by the effective dimension. In particular, we get the following error estimates from the above theorem under the different behavior of the effective dimension (Assumptions~\ref{N(l).bound}, \ref{log.decay}):

\begin{corollary}\label{err.upper.bound.p.para}
Under the same assumptions of Theorem \ref{err.upper.bound.p} and Assumption \ref{N(l).bound} on effective dimension~$\mathcal{N}(\la)$ with the a-priori choice of the regularization parameter~$\la=m^{-\frac{p+1}{p+q+b(p+1)}}$, for all~$0<\eta<1$, the following error estimate holds with confidence~$1-\eta$:
$$||\fz -\fp||_\HH = \mathcal{O}\paren{m^{-\frac{q}{2(p+q)+2b(p+1)}}\log^2\left(\frac{4}{\eta}\right)}.$$
\end{corollary}

\begin{corollary}\label{err.upper.bound.cor.log}
Under the same assumptions of Theorem \ref{err.upper.bound.p} and Assumption \ref{log.decay} on effective dimension~$\mathcal{N}(\la)$ with the a-priori choice of the regularization parameter~$\la=\left(\frac{\log m}{m}\right)^{\frac{p+1}{p+q}}$, for all~$0<\eta<1$, we have the following convergence rate with confidence~$1-\eta$:
$$||\fz-\fp||_\HH = \mathcal{O}\paren{\left(\frac{\log m}{m}\right)^{\frac{q}{2(p+q)}}\log^2\left(\frac{4}{\eta}\right)}.$$
\end{corollary}

Notice that the convergence rate given in the Corollary~\ref{err.upper.bound.para} is worse than the one  in the Corollary~\ref{err.upper.bound.p.para} since we use the rough estimate \eqref{Nl.bd} for the effective dimension in Corollary~\ref{err.upper.bound.para}.   

\section{Discussion}\label{Discussion}
We discussed a finite sample bound of Tikhonov regularization scheme for nonlinear statistical inverse problems in vector-valued setting, therefore the results can be applied to the multitask learning problem. The convergence rates presented in Section~\ref{Sec-analysis} hold asymptotically, i.e., all parameters are fixed as~$m\to\infty$.  The considered framework covers previously proposed settings for different learning schemes: direct, linear inverse learning problems.


The rates of convergence were represented in terms of the effective dimension~$\mathcal{N}(\la)$ of the governing operator~$\tp$ which can be seen from the basic probabilistic bound, given in Proposition~\ref{main.bound}. Also, the Corollaries~\ref{err.upper.bound.p.para} and~\ref{err.upper.bound.cor.log} can be given a handy representation of the error bounds under different behavior of the effective dimension corresponding to the ill-posedness of the problem.


This is well-known that Tikhonov regularization suffers the saturation effect. We observe from the analysis in Section~\ref{Sec-analysis} that using the Tikhonov regularization in Hilbert scales the saturation effect can be ignored. 

The a-priori parameter choice considered in our analysis requires the knowledge of the parameters~$b$,~$p$,~$q$, which is typically unknown in practice. In practice, a posteriori parameter choice rule (data-dependent) for the regularization parameter~$\la$ such as the Lepskii-balancing principle, discrepancy principle, quasi-optimality principle with theoretical justification need to be considered so that we can turn our results to data-dependent minimax adaptivity without a priori knowledge of the regularity parameters.


\appendix

\section{Probabilistic estimates and preliminaries results}

%

The following bounds are standard in learning theory, in which we estimate the effect of random sampling using Assumption~\ref{noise.cond} in the probabilistic sense. The following propositions can be proved similar to the arguments given in~\cite[Theorem~4]{Caponnetto}.


\begin{proposition}\label{main.bound}
Suppose Assumptions~\ref{assmpt1}--\ref{noise.cond}  hold true, then for~$m \in \NN$ and~$0<\eta<1$, each of the following estimate holds with the confidence~$1-\eta$,
\begin{equation*}
\norm{(\tp +\la I)^{-1/2}\sx^*\brac{\sx A(\fp )-\yy}}_{\HH'} \leq 2\paren{\frac{\kappa M}{m\sqrt{\la}}+\sqrt{\frac{\Sigma^2\mathcal{N}(\la)}{m}}}\log\left(\frac{2}{\eta}\right),
\end{equation*}

and
\begin{equation*}
\|(\tp +\la I)^{-1/2}(\tx  -\tp )\|_{\mathcal{L}_2(\HH')}\leq 2\left(\frac{\kappa^2}{m\sqrt{\la}}+\sqrt{\frac{\kappa^2\mathcal{N}(\la)}{m}}\right)\log\left(\frac{2}{\eta}\right).
\end{equation*}

\end{proposition}

In the following proposition, the probabilistic estimate of the first term can be established under the condition~\eqref{l.la.condition} on the regularization parameter~$\la$ and sample size~$m$. Then we obtain the last two estimates using \cite[Prop.~A.2]{Blanchard2019}.

\begin{proposition}\label{I1}
Suppose Assumption~\ref{assmpt1} and the condition~\eqref{l.la.condition} hold true, then for~$m \in \NN$ and~$0<\eta<1$, each of the following estimates hold with the confidence~$1-\eta/2$,

\begin{equation*}
\norm{(\tp+\la I)^{-\frac{1}{2}}(\tp-\tx)}_{\mathcal{L}_2(\HH')}\leq \sqrt{\la}2\kappa(2\kappa+1)\log\paren{\frac{2}{\eta}},
\end{equation*}
and for~$0\leq s \leq 1$,
\begin{align*}
\norm{(\tx+\la I)^{-s}(\tp+\la I)^s}_{\mathcal{L}(\HH')}\leq  \paren{\frac{\Psi_\xx}{\sqrt{\la}}+1}^{2s} \leq & \paren{(2\kappa+1)^2\log\paren{\frac{2}{\eta}}}^{2s}.
\end{align*}
\end{proposition}

\begin{lemma}\label{R_la}
Let Assumption~\ref{A.assumption} holds true. Then for~$f \in \mathcal{B}_d(\fp ) \cap \mathcal{D}(A) \subset \HH$, the error terms can be estimated as
\begin{align*}
\norm{A(f)-A(\fp)-A'(\fp)(f-\fp)}_{\HH'} \leq 2 J \norm{f-\fp}_{\HH}
\end{align*}
and
\begin{align*}
\norm{\ip  \brac{A(f)-A(\fp)-A'(\fp)(f-\fp)}}_{\LL} \leq \frac{\gamma}{2} \norm{f-\fp}_{\HH}\norm{f-\fp}_{\HH_{-p}}.
\end{align*}
\end{lemma}
\begin{proof}
Under the Assumption~\ref{A.assumption} we obtain
\begin{align*}
\norm{A(f)-A(\fp)-A'(\fp)(f-\fp)}_{\HH'}    
=& \norm{\int_{0}^1 \brac{A'\paren{\fp+t(f-\fp)}-A'(\fp)}\paren{f-\fp}dt}_{\HH'}\\ \nonumber
\leq& 2J \norm{f-\fp}_{\HH}
\end{align*}
and
\begin{align*}
& \norm{\ip  \brac{A(f)-A(\fp)-A'(\fp)(f-\fp)}}_{\LL}    \\ \nonumber
=& \norm{\int_{0}^1 \ip  \brac{A'\paren{\fp+t(f-\fp)}-A'(\fp)}\paren{f-\fp}dt}_{\LL}\\ \nonumber
\leq& \int_{0}^1\norm{\ip  \brac{A'\paren{\fp+t(f-\fp)}-A'(\fp)}}_{ \HH_{-p}\to\LL}\norm{f-\fp}_{\HH_{-p}}dt \\ \nonumber
\leq& \int_{0}^1\gamma t\norm{f-\fp}_{\HH}\norm{f-\fp}_{\HH_{-p}}dt  \\ \nonumber
=& \frac{\gamma}{2} \norm{f-\fp}_{\HH}\norm{f-\fp}_{\HH_{-p}}.
\end{align*}
\end{proof}

\begin{lemma}\label{B3}
For the error term in eq. \eqref{Taylor_exp_Rla} under the Assumption~\ref{A.assumption} we have:
\begin{align*}
&\norm{\Delta_{\zz}}_{m}^2-\norm{\Delta_{\zz}+\sx  r(\fz)}_{m}^2  
 \leq  \Gamma_{\xx}^2\Theta_{\zz}^2+4J\sqrt{\la}\Gamma_{\xx}\Theta_{\zz}\norm{\fz-\fp}_{\HH},
\end{align*}
\begin{align*}
&\inner{A'(\fp)(\fz-\fp),(\tp -\tx  )A'(\fp)(\fz-\fp)-2\sx ^*\Delta_{\zz}+2(\tp -\tx  )r(\fz )}_{\HH'}\\
\leq &\brac{2\Theta_{\zz}+5\Psi_\xx J\norm{\fz-\fp}_{\HH}}\brac{\sqrt{\la} J\norm{\fz-\fp}_{\HH}+\beta\norm{\fz-\fp}_{\HH_{-p}}}
\end{align*}
and
\begin{align*}
\inner{A'(\fp)(\fz-\fp),\tp   r(\fz )}_{\HH'}
\leq \frac{\beta\gamma}{2}\norm{\fz-\fp}_{\HH}\norm{\fz-\fp}_{\HH_{-p}}^2.
\end{align*}
\end{lemma}
\begin{proof}
From Lemma~\ref{R_la} and~\eqref{L.op} under Assumption~\ref{A.assumption} we have,
\begin{align*}
&\norm{\Delta_{\zz}}_{m}^2-\norm{\Delta_{\zz}+\sx  r(\fz)}_{m}^2 \\   \nonumber
= & \norm{\Xi_{\xx}\Delta_{\zz}}_m^2+2\inner{\Xi_{\xx}\Delta_{\zz},(I-\Xi_{\xx})\Delta_{\zz}}_{m}-\norm{\Xi_{\xx}\Delta_{\zz}+\sx  r(\fz)}_m^2-2\inner{\Xi_{\xx}\Delta_{\zz}+\sx  r(\fz),(I-\Xi_{\xx})\Delta_{\zz}}_m\\   \nonumber
\leq & \Gamma_{\xx}^2\Theta_{\zz}^2-2\inner{\sx  r(\fz),(I-\Xi_{\xx})\Delta_{\zz}}_{m}
=  \Gamma_{\xx}^2\Theta_{\zz}^2-2\la\inner{r(\fz),(\tx  +\la I)^{-1}\sx ^*\Delta_{\zz}}_{\HH'}\\  \nonumber
 \leq & \Gamma_{\xx}^2\Theta_{\zz}^2+2\sqrt{\la}\Gamma_{\xx}\Theta_{\zz}\norm{r(\fz)}_{\HH'}  
 \leq  \Gamma_{\xx}^2\Theta_{\zz}^2+4J\sqrt{\la}\Gamma_{\xx}\Theta_{\zz}\norm{\fz-\fp}_{\HH}
\end{align*}
and using the inequality 
\begin{align*}
\inner{f,g}_{\HH'}= & \la\inner{f,(\tp +\la I)^{-1}g}_{\HH'}+\inner{f,\tp (\tp +\la I)^{-1}g}_{\HH'}\\
\leq & \brac{\sqrt{\la}\norm{f}_{\HH'}+\norm{\ip   f}_{\LL}}\norm{(\tp +\la I)^{-1/2}g}_{\HH'}
\end{align*}
we obtain,
\begin{align*}
&\inner{A'(\fp)(\fz-\fp),(\tp -\tx  )A'(\fp)(\fz-\fp)-2\sx ^*\Delta_{\zz}+2(\tp -\tx  )r(\fz )}_{\HH'}\\  \nonumber
\leq & \brac{\Psi_\xx\norm{A'(\fp)(\fz-\fp)}_{\HH'}+2\Theta_{\zz}+2\Psi_\xx\norm{r(\fz)}_{\HH'}}\\
&\times\brac{\sqrt{\la}\norm{A'(\fp)(\fz-\fp)}_{\HH'}+\norm{\ip  A'(\fp)(\fz-\fp)}_{\LL}}\\   \nonumber
\leq &\brac{2\Theta_{\zz}+5\Psi_\xx J\norm{\fz-\fp}_{\HH}}\brac{\sqrt{\la} J\norm{\fz-\fp}_{\HH}+\beta\norm{\fz-\fp}_{\HH_{-p}}}
\end{align*}
and
\begin{align*}
&\inner{A'(\fp)(\fz-\fp),\tp   r(\fz )}_{\HH'}\\  \nonumber
\leq &\inner{\ip  A'(\fp)(\fz-\fp),\ip  r(\fz)}_{\LL} \\  \nonumber
\leq & \norm{\ip  r(\fz)}_{\LL}\norm{\ip  A'(\fp)(\fz-\fp)}_{\LL} \\  \nonumber
\leq & \frac{\beta\gamma}{2}\norm{\fz-\fp}_{\HH}\norm{\fz-\fp}_{\HH_{-p}}^2.
\end{align*}
\end{proof}

\section*{Acknowledgments} The author is grateful to G. Blanchard and P. Math{\'e} for useful discussions and suggestions. This research has been partially funded by Deutsche Forschungsgemeinschaft (DFG) through grant CRC 1294 ``Data Assimilation", Project (A04) ``Nonlinear statistical inverse problems with random observations". 

\bibliography{library}
\bibliographystyle{abbrv}
\end{document}